\documentclass[english]{smfart}
\usepackage{brieskorn}

\begin{document}
\frontmatter
\title{Some properties and applications of Brieskorn~lattices}

\author[C.~Sabbah]{Claude Sabbah}
\address[C.~Sabbah]{CMLS, École polytechnique, CNRS, Université Paris-Saclay\\
F--91128 Palaiseau cedex\\
France}
\email{Claude.Sabbah@polytechnique.edu}
\urladdr{http://www.math.polytechnique.fr/perso/sabbah}

\begin{abstract}
After reviewing the main properties of the Brieskorn lattice in the framework of tame regular functions on smooth affine complex varieties, we prove a conjecture of Katzarkov-Kontsevich-Pantev in the toric case.
\end{abstract}


\maketitle
\tableofcontents
\mainmatter

\section{Introduction}
The \emph{Brieskorn lattice}, introduced by Brieskorn in \cite{Brieskorn70} in order to provide an algebraic computation of the Milnor monodromy of a germ of complex hypersurface with an isolated singularity, has also proved central in the Hodge theory for vanishing cycles of such a singularity, as emphasized by Pham \cite{Pham79, Pham83b}. Hodge theory for vanishing cycles, as developed by Steenbrink \cite{Steenbrink76,Steenbrink77,S-S85} and Varchenko \cite{Varchenko82}, makes it an analogue of the Hodge filtration in this context, and fundamental results have been obtained by M.\,Saito \cite{MSaito89} in order to characterize it among other lattices in the Gauss-Manin system of an isolated singularity of complex hypersurface. As such, it leads to the definition of a period mapping, as introduced and studied with much detail by K.\,Saito for some singularities \cite{KSaito83b}. It is also a basic constituent of the period mapping restricted to the $\mu$-constant stratum \cite{MSaito91}, where a natural Torelli problem occurs (\cf \cite{MSaito91}, \cite{Hertling99b}).

For a holomorphic germ $f:(\CC^{n+1},0)\to(\CC,0)$ with an isolated singularity, denoting by $t$ the coordinate on the target space $\CC$, the space
\begin{equation}\label{eq:Brlat}
\Omega^{n+1}_{\CC^{n+1},0}/\rd f\wedge\rd\Omega^{n-1}_{\CC^{n+1},0}
\end{equation}
is naturally endowed with a $\CC\{t\}$-module structure (where $t$ acts as the multiplication by $f$), and the \emph{Brieskorn lattice} is the $\CC\{t\}$-module (\cf\cite[p.\,125]{Brieskorn70})
\begin{equation}\label{eq:Brlatt}
{}''\!H^n_{f,0}=\Big(\Omega^{n+1}_{\CC^{n+1},0}/\rd f\wedge\rd\Omega^{n-1}_{\CC^{n+1},0}\Big)\Big/\CC\{t\}\text{-torsion}.
\end{equation}
Brieskorn shows that \eqref{eq:Brlatt} is free of finite rank equal to the Milnor number $\mu(f,0)$, and Sebastiani \cite{Sebastiani70} shows the torsion freeness of \eqref{eq:Brlat}, which can thus also serve as an expression for ${}''\!H^n_{f,0}$. It is also endowed with a meromorphic connection~$\nabla$ having a pole of order at most two at $t=0$, and the $\CC(\!\{t\}\!)$-vector space with connection generated by ${}''\!H^n_{f,0}$ is isomorphic to the Gauss-Manin connection, which has a regular singularity there. ${}''\!H^n_{f,0}$ is thus a $\CC\{t\}$-lattice of this $\CC(\!\{t\}\!)$-vector space. While the action of $\nabla_{\partial_t}$, simply written as $\partial_t$, introduces a pole, there is a well-defined action of its inverse $\partial_t^{-1}$ that makes ${}''\!H^n_{f,0}$ a module over the ring of $\CC\{\!\{\partial_t^{-1}\}\!\}$ of $1$-Gevrey series (\ie formal power series $\sum_{n\geq0}a_n\partial_t^{-n}$ such that the series $\sum_na_nu^n/n!$ converges). It happens to be also free of rank $\mu$ over this ring (\cite{Malgrange74b,Malgrange75b}). The relation between the rings $\CC\{t\}$ and $\CC\{\!\{\partial_t^{-1}\}\!\}$ is called \emph{microlocalization}. In the global case below, we will use instead the Laplace transformation. The mathematical richness of this object leads to various generalizations.

For non-isolated hypersurface singularities, the objects with definition as in \eqref{eq:Brlatt} (but in various degrees) have been introduced by Hamm in his Habilitationsschrift (\cf \cite[\S II.5]{Hamm75b}), who proved that they are $\CC\{t\}$-free of finite rank, but do not coincide with \eqref{eq:Brlat} in general. A natural $\CC\{\!\{\partial_t^{-1}\}\!\}$-structure still exists on \eqref{eq:Brlat}, and Barlet and Saito \cite{B-S07} have shown that the $\CC\{t\}$-torsion and the $\CC\{\!\{\partial_t^{-1}\}\!\}$-torsion coincide, so that ${}''\!H^k_{f,0}$ remains $\CC\{\!\{\partial_t^{-1}\}\!\}$-free of finite rank.

The Brieskorn lattice has also a global variant. On the one hand, the Brieskorn lattice for tame regular functions on smooth affine complex varieties (\cf Section \ref{sec:tame}) is a direct analogue of the case of an isolated singularity, but the double pole of the action of $t$ with respect to the variable $\partial_t^{-1}$ cannot in general be reduced to a simple one by a meromorphic (even formal) gauge transformation \ie the Gauss-Manin system with respect to the variable $\partial_t^{-1}$ has in general an irregular singularity. The properties of the Brieskorn module for regular functions on affine manifolds which are not tame have been considered by Dimca and M.\,Saito \cite{D-S01}.

On the other hand, given a \emph{projective} morphism $f:X\to\Afu$ on a smooth quasi-projective variety $X$, the Brieskorn modules, defined as the hypercohomology $\CC[\partial_t^{-1}]$-modules of the twisted de~Rham complex \hbox{$(\Omega_X^\cbbullet[\partial_t^{-1}],\rd-\partial_t^{-1}\rd f)$}, have been shown to be $\CC[\partial_t^{-1}]$-free (Barannikov-Kontsevich, \cf\cite{Bibi97b}), and a similar result holds when one replaces $\Omega_X^\cbbullet$ with $\Omega_X^\cbbullet(\log D)$ for some divisor with normal crossings. More generally, one can adapt the definition of the Brieskorn modules for the twisted de~Rham complex attached to a mixed Hodge module, and the $\CC[\partial_t^{-1}]$-freeness still holds, so that they can be called Brieskorn lattices (\cf\loccit). This enables one to use the push-forward operation by the map $f$ and reduce the study to that of Brieskorn lattices attached to mixed Hodge modules on the affine line, as for example the mixed Hodge modules that the Gauss-Manin systems of $f$ underlie. In such a way, the Brieskorn lattice has a \emph{purely Hodge-theoretic definition}, which does not refer to the underlying geometry, and can thus be attached, for example, to any polarizable variation of Hodge structure on a punctured affine line (\cf \cite[\S1.d]{Bibi05}).

The Brieskorn lattice of tame functions is of particular interest and has been considered in \cite{Bibi96bb} for example. The Brieskorn lattice for families of such functions, considered in \cite{D-S02a}, has been investigated with much care for families of Laurent polynomials in relation with mirror symmetry by Reichelt and Reichelt-Sevenheck \cite{R-S10,Reichelt14,Reichelt15,R-S17}.

Lastly, in the global setting as above, the pole of order two of the action of $t$ with respect to the variable $\partial_t^{-1}$ produces in general a truly irregular singularity, and the Brieskorn lattice is an essential tool to produce the \emph{irregular Hodge filtration} attached to such a singularity (\cf \cite{S-Y14,Bibi15}).

The contents of this article is as follows. In Section \ref{sec:tame}, we review known results on the Brieskorn lattice for a tame function. We show in Section \ref{sec:KKP} how these results enables one to obtain a simple proof of a conjecture of Katzarkov-Kontsevich-Pantev in the toric case.

\subsubsection*{Acknowledgements}
I thank the referee for his/her careful reading of the manuscript and  interesting suggestions and Claus Hertling for pointing out Lemma \ref{lem:CH}.

\section{The Brieskorn lattice of a tame function}\label{sec:tame}

In this section, we review the main properties of the Brieskorn lattice attached to a tame function on an affine manifold, following \cite{Bibi96b, Bibi96bb,D-S02a}.

Let $U$ be a smooth complex affine variety of dimension $n$ and let $\fun\in\cO(U)$ be a regular function on $U$. There are various notions of tameness for such a function, which are not known to be equivalent, but for what follows they have the same consequences. One of the definitions, given by Katz in \cite[Th.\,14.13.3]{Katz90}, is that the cone of $f_!\CC_U\to\bR f_*\CC_U$ should have constant cohomology on $\Afu$. We will use the notion of a weakly tame function, as defined in \cite{N-S97}, that is, either cohomologically tame or M-tame. 

We assume that $\fun$ is weakly tame. Let $\theta$ be a new variable. The \emph{Brieskorn lattice} attached to $\fun$ is the $\CC[\theta]$-module
\[
G_0:=\Omega^n(U)[\theta]\big/(\theta\rd-\rd\fun)\Omega^{n-1}(U)[\theta].
\]
An expression like \eqref{eq:Brlat} also exists if $U$ is the affine space $\mathbb{A}^{n+1}$, but the above one is valid for any smooth affine variety~$U$. The variable $\theta$ is for $\partial_t^{-1}$. We already notice that
\begin{equation}\label{eq:G0}
G_0/\theta G_0\simeq\Omega^n(U)/\rd\fun\wedge\Omega^{n-1}(U)
\end{equation}
has dimension equal to the sum $\mu=\mu(\fun)$ of the Milnor numbers of $\fun$ at all its critical points in $U$. The following properties are known in this setting.

\begin{enumerate}
\item\label{enum:tame1}
The algebraic Gauss-Manin systems $\cH^k\fun_+\cO_U$ are isomorphic to powers of the $\Clt$-module $(\CC[t],\partial_t)$, except for $k=0$, so their localized Laplace transforms vanish except that for $k=0$. If we regard the Laplace transform of $\cH^0\fun_+\cO_U$ as a $\Cltau$-module, we know that it has finite type as such, and its localized Laplace transform $G$, that is, the $\CC[\tau,\tau^{-1}]$-module obtained by localization, is free of rank $\mu$. We have
\[
G=\Omega^n(U)[\tau,\tau^{-1}]\big/(\rd-\tau\rd\fun)\Omega^{n-1}(U)[\tau,\tau^{-1}].
\]
\item\label{enum:tame2}
Setting $\theta=\tau^{-1}$, we write
\[
G=\Omega^n(U)[\theta,\theta^{-1}]\big/(\theta\rd-\rd\fun)\Omega^{n-1}(U)[\theta,\theta^{-1}],
\]
and there is therefore a natural morphism $G_0\to G$. This morphism is \emph{injective}, so that $G_0$ is a \emph{free} $\CC[\theta]$-module of rank $\mu$ such that $\CC[\theta,\theta^{-1}]\otimes_{\CC[\theta]}G_0=G$, \ie $G_0$ is a $\CC[\theta]$-lattice of $G$, on which the restriction of the Gauss-Manin connection has a pole of order at most two. Moreover, the action of $\theta^2\partial_\theta$ on the class $[\omega]$ of $\omega\in\Omega^n(U)$ in~$G_0$ is given~by\enlargethispage{\baselineskip}%
\[
\theta^2\partial_\theta[\omega]=[\fun\omega],
\]
and the action of $\theta^2\partial_\theta$ on a polynomial $\sum_{k\geq0}[\omega_k\theta^k]$ is obtained by the usual formulas.

\item\label{enum:tame3}
Let $V_\bbullet G$ be the (increasing) $V$-filtration of $G$ with respect to the function~$\tau$ (recall that $G$ has a regular singularity at $\tau=0$, while that at infinity is usually irregular). It is a filtration by free $\CC[\tau]$-modules of rank $\mu$ indexed by $\QQ$. The jumping indices of the induced filtration $V_\bbullet (G_0/\theta G_0)$, together with their multiplicities (the dimension of $\gr_\beta^V(G_0/\theta G_0)$) form the \emph{spectrum of $\fun$ at $\infty$}. The jumping indices are contained in the interval $[0,n]\cap\QQ$ and the spectrum is symmetric with respect to~$n/2$. 
\item\label{enum:tame4}
On the other hand, for $\alpha\in[0,1)\cap\QQ$, the vector space $\gr_\alpha^VG$ is endowed with the nilpotent endomorphism $\rN$ induced by the action of $-(\tau\partial_\tau+\alpha)$ and with the increasing filtration $G_\bbullet\gr_\alpha^VG$ naturally induced by the filtration $G_p=\theta^{-p}G_0$, \ie
\[
G_p\gr_\alpha^VG=(G_p\cap V_\alpha G)/(G_p\cap V_{<\alpha}G),
\]
where the intersections are taken in $G$. As a consequence, we have isomorphisms ($p\in\ZZ$, $\alpha\in[0,1)$)
\[
\gr_p^G\gr_\alpha^VG\xrightarrow[\sim]{\textstyle~\theta^p~}\gr_{\alpha+p}^V(G_0/\theta G_0).
\]
\item\label{enum:tame5}
The $\CC$-vector space $H_{\neq1}\!:=\!\bigoplus_{\alpha\in(0,1)\cap\QQ}\gr_\alpha^VG$, \resp $H_1\!:=\!\gr_0^VG$, endowed with
\begin{itemize}
\item
the filtration
\[
F^pH_{\neq1}:=\bigoplus_{\alpha\in(0,1)\cap\QQ}G_{n-1-p}\gr_\alpha^VG\quad \resp F^pH_1=G_{n-p}\gr_0^VG,
\]

\item
and the weight filtration $W_\bbullet=\rM(\rN)[n-1]$ (\resp $\rM(\rN)[n]$), \ie the monodromy filtration of $\rN$ centered at $n-1$ (\resp $n$),
\end{itemize}
is part of a mixed Hodge structure. In particular, $\rN$ strictly shifts by one the filtration $G_\bbullet\gr_\alpha^VG$ and acts on the graded space $\gr_\bbullet^G\gr_\alpha^VG$ as the degree-one morphism induced by $-\tau\partial_\tau$. We therefore have a commutative diagram, for any $\alpha\in[0,1)$ and $p\in\ZZ$, (\cf \cite{Varchenko81} and \cite[\S7]{S-S85} in the singularity case):
\begin{equation}\label{eq:diag}
\begin{array}{c}
\xymatrix@C=1.2cm@R=.7cm{
\gr_p^G\gr_\alpha^VG\ar[r]^-{\theta^p}_-\sim\ar[d]_{[\rN]}&\gr_{\alpha+p}^V(\Omega^n(U)/\rd\fun\wedge\Omega^{n-1}(U))\ar[d]^{[\fun]}\\
\gr_{p+1}^G\gr_\alpha^VG\ar[r]^-{\theta^{p+1}}_-\sim&\gr_{\alpha+p+1}^V(\Omega^n(U)/\rd\fun\wedge\Omega^{n-1}(U)),
}
\end{array}
\end{equation}
by using the relation $-\tau\partial_\tau=\theta\partial_\theta$.
\par
To see this, write the commutative diagram
\[
\xymatrix@C=1.2cm@R=1cm{
\gr_p^G\gr_\alpha^VG\ar[r]^-{\theta^p}_-\sim\ar[d]_{\theta\partial_\theta-\alpha}&\gr_{\alpha+p}^V\gr_0^GG\ar[d]_{\theta\partial_\theta-(\alpha+p)}\ar[rd]\\
\gr_{p+1}^G\gr_\alpha^VG\ar[r]^-{\theta^p}_-\sim&\gr_{\alpha+p}^V\gr_1^GG\ar[r]^-\theta&\gr_{\alpha+p+1}^V\gr_0^GG
}
\]
and use that in the vertical morphisms, the constant part $\alpha$ or $\alpha+p$ induces the morphism $0$.

\item\label{enum:tame5b}
Recall that a mixed Hodge structure $(H_\QQ,F^\cbbullet H_\CC,W_\bbullet H_\QQ)$ is said to be \emph{of Hodge-Tate type} if
\begin{enumerate}
\item\label{enum:tame5ba}
the filtration $W_\bbullet$ has only even jumping indices
\item\label{enum:tame5bb}
and $W_{2\bbullet}H_\CC$ is opposite to $F^\cbbullet H_\CC$. 
\end{enumerate}
The description of the mixed Hodge structure given in \eqref{enum:tame5} implies the following criterion. We will set $\nu=n-1$ when considering $H_{\neq1}$ and $\nu=n$ when considering $H_1$. We will then denote by $H$ either $H_{\neq1}$ or $H_1$.

\begin{corollaire}\label{cor:criterionHT}
The mixed Hodge structure that the triple $(H,F^\cbbullet H,W_\bbullet H)$ underlies is of Hodge-Tate type if and only if, for any integer $k$ such that $0\leq k\leq[\nu/2]$, the $(\nu-2k)$th power of $\rN$ induces an isomorphism
$$
[\rN]^{\nu-2k}:\gr_k^GH\isom\gr_{\nu-k}^GH.
$$
\end{corollaire}

\begin{proof}
We define the filtration $W'_\bbullet H$ indexed by $2\ZZ$ by the formula $W'_{2k}H=G_{\nu-k}H$, so that $G_kH=W'_{2(\nu-k)}H$. If we set $\ell=\nu-2k$ for $0\leq k\leq\nu/2$, we have $0\leq\ell\leq\nu$ and the isomorphism in the corollary is written
\[
[\rN]^\ell:\gr_{\nu+\ell}^{W'}H\isom\gr_{\nu-\ell}^{W'}H.
\]
We can conclude that $W'_\bbullet H=W_\bbullet H$ if we know that $\rN^{\nu+1}=0$, that is, $\gr_{\nu+1}^GH=0$. This is a consequence of the positivity of the spectrum \cite[Cor.\,13.2]{Bibi96bb}, which says that, if $\alpha\in[0,1)$, we have $\gr_k^G\gr_\alpha^VG=0$ for $k\notin[0,\nu]\cap\NN$. 
\end{proof}

The following lemma was pointed out to me by Claus Hertling.

\begin{lemme}\label{lem:CH}
A mixed Hodge structure $(H_\QQ,F^\cbbullet H_\CC,W_\bbullet H_\QQ)$ is Hodge-Tate if and only if we have, for all $p\in\tfrac12\ZZ$,
\[
\dim\gr^p_FH_\CC=\dim\gr_{2p}^WH_\QQ.
\]
\end{lemme}

\begin{proof}
Indeed, one direction is clear. Conversely, if the equality of dimensions holds, then \eqref{enum:tame5ba} holds since $F^\cbbullet H$ has only integral jumps; moreover, up to a Tate twist, one can assume that $W_{<0}H=0$, so $\gr^k_FH=0$ for $k<0$. It is enough to prove that $\gr^p_F\gr_{2\ell}^WH=0$ for all $p\neq\ell$. We prove this by induction on $\ell$. If $\ell=0$, the result follows from the property that $F^pH=0$ for $p<0$ and Hodge symmetry. Assume the result up to $\ell$. For $j\leq\ell$ we thus have $\dim\gr^j_F\gr_{2j}^WH=\dim\gr_{2j}^WH=\dim\gr^j_FH$ (the latter equality by the assumption), and therefore $\gr_{2i}^W\gr^j_FH=0$ for $i\neq j$. In particular, taking $i=\ell+1$, we have $\gr^p_F\gr_{2(\ell+1)}^WH=0$ for all $p\leq\ell$. By Hodge symmetry, we obtain $\gr^p_F\gr_{2(\ell+1)}^WH=0$ for all $p\neq\ell+1$, as wanted.
\end{proof}

\item\label{enum:tame6}
We now consider the case where $U=(\CC^*)^n$, endowed with coordinates $x=(x_1,\dots,x_n)$. Let $\fun\in\CC[x,x^{-1}]$ be a Laurent polynomial in $n$ variables, with Newton polyhedron $\Delta(\fun)$. We assume that $\fun$ is \emph{nondegenerate with respect to its Newton polyhedron} and \emph{convenient} (\cf \cite{Kouchnirenko76}). In particular, $0$ belongs to the interior of its Newton polyhedron. It is known that such a function is M-tame.

For any face $\sigma$ of dimension $n-1$ of the boundary $\partial\Delta(f)$, we denote by $L_\sigma$ the linear form with coefficients in $\QQ$ such that $L_\sigma\equiv1$ on $\sigma$. For $g\in\CC[x,x^{-1}]$, we set $\deg_\sigma(g)=\max_mL_\sigma(m)$, where the max is taken on the exponents of monomials $x^m$ appearing in~$g$, and $\deg_{\Deltas}(g)=\max_\sigma\deg_\sigma(g)$. We denote the volume form $\rd x_1/x_1\wedge\cdots\wedge \rd x_n/x_n$ by $\omega$, giving rise to an identification $\CC[x,x^{-1}]\isom\Omega^n(U)$ and $\CC[x,x^{-1}]/(\partial f)\isom G_0/\theta G_0$ (\cf\eqref{eq:G0}).

The Newton increasing filtration $\ccN_\bbullet\Omega^n(U)$ indexed by $\QQ$ is defined by
\[
\ccN_\beta\Omega^n(U):=\{g\omega\in\Omega^n(U)\mid\deg_{\Deltas}(g)\leq\beta\}.
\]
We have $\ccN_\beta\Omega^n(U)=0$ for $\beta<0$ and $\ccN_0\Omega^n(U)=\CC\cdot \omega$. We can extend this filtration to $\Omega^n(U)[\theta]$ by setting
\[
\ccN_\beta\Omega^n(U)[\theta]\defin\ccN_\beta\Omega^n(U)+\theta\ccN_{\beta-1}\Omega^n(U)+\cdots+\theta^k\ccN_{\beta-k}\Omega^n(U)+\cdots
\]
and then naturally induce this filtration on $G_0$, to obtain a filtration $\ccN_\bbullet G_0$ and then on $G_0/\theta G_0$. We have
\begin{equation}
\ccN_\bbullet G_0=V_\bbullet G\cap G_0\quad\text{and}\quad \ccN_\bbullet(G_0/\theta G_0)=V_\bbullet(G_0/\theta G_0).
\end{equation}
Corollary \ref{cor:criterionHT} now reads, according to \eqref{eq:diag} and by using the above identification through multiplication by~$\omega$:

\begin{corollaire}\label{cor:criterionHTLaurent}
The mixed Hodge structure that the triple $(H,F^\cbbullet H,W_\bbullet H)$ underlies is of Hodge-Tate type if and only if, for any integer $k$ such that $0\leq k\leq[\nu/2]$ ($\nu=n-1$, \resp $n$), we have isomorphisms
\begin{align*}
\gr_{\alpha+k}^\ccN\bigl(\CC[x,x^{-1}]/(\partial f)\bigr)&\xrightarrow[\textstyle\sim]{\textstyle[f]^{n-1-2k}}\gr_{\alpha+n-1-k}^\ccN\bigl(\CC[x,x^{-1}]/(\partial f)\bigr)\quad\forall\alpha\in(0,1),\\
\tag*{resp.}\gr_k^\ccN\bigl(\CC[x,x^{-1}]/(\partial f)\bigr)&\xrightarrow[\textstyle\sim]{\textstyle[f]^{n-2k}}\gr_{n-k}^\ccN\bigl(\CC[x,x^{-1}]/(\partial f)\bigr).
\end{align*}
\end{corollaire}
\end{enumerate}

\section{On a conjecture of Katzarkov-Kontsevich-Pantev}\label{sec:KKP}
In this section we use the algebraic Brieskorn lattice of a convenient nondegenerate Laurent polynomial to solve the toric case of the part ``$f^{p,q}=h^{p,q}$'' of Conjecture 3.6 in \cite{K-K-P14} (the other equality ``$h^{p,q}=i^{p,q}$'' is obviously not true by simply considering the case of the standard Laurent polynomial mirror to the projective space $\PP^n$, see also another counter-example in \cite{L-P16}). We refer to \cite{L-P16,Harder17,Shamoto17} for further discussion and positive results on this conjecture.

\subsection{The Brieskorn lattice and the conjecture of Katzarkov-Kontsevich-Pantev}
Given a smooth quasi-projective variety $U$ and a morphism $\fun:U\!\to\!\Afu$, \hbox{every} twisted de~Rham cohomology $H^k_{\DR}(U,\rd+\rd \fun)$, \ie the $k$th hypercohomology of the twisted de~Rham complex $(\Omega_U^\cbbullet,\rd+\rd \fun)$, is endowed with a decreasing filtration $F_{\Yu}^\cbbullet H^k_{\DR}(U,\rd+\rd \fun)$ indexed by $\QQ$ (\cf \cite{Yu12}). For $\alpha\in[0,1)$, the filtration indexed by $\ZZ$ defined by $F_{\Yu,\alpha}^p=F_\Yu^{p-\alpha}$ can also be computed in terms of the Kontsevich complex $\Omega^\cbbullet_\fun(\alpha)$ together with its stupid filtration (\cf\cite[Cor.\,1.4.5]{E-S-Y13}). The irregular Hodge numbers $h_\alpha^{p,q}(\fun)$ are defined as
\begin{equation}\label{eq:alpha}
h_\alpha^{p,q}(\fun):=\dim\gr^{p-\alpha}_{F_\Yu}H^{p+q}_{\DR}(U,\rd+\rd \fun).
\end{equation}
It is well-known that $\dim H^k_{\DR}(U,\rd+\rd \fun)=\dim H^k(U,\fun^{-1}(t))$ for $|t|\gg0$. This space is endowed with a monodromy operator (around $t=\infty$), and we will consider the case where this monodromy operator is \emph{unipotent}. In such a case, the filtration $F_{\Yu}^\cbbullet H^{p+q}_{\DR}(U,\rd+\rd \fun)$ is known to jump at integers only, and in \eqref{eq:alpha} only $\alpha=0$ occurs. We then simply denote this number by $h^{p,q}(\fun)$, so that, in such a case,
\[
h^{p,q}(\fun):=\dim\gr^p_{F_\Yu}H^{p+q}_{\DR}(U,\rd+\rd \fun).
\]
Let $W_\bbullet$ be the monodromy filtration on $H^k(U,\fun^{-1}(t))$ centered at $k$. The conjecture of \cite{K-K-P14} that we consider is the possible equality (\cf \cite{L-P16,Harder17,Shamoto17})
\begin{equation}\label{eq:conjKKP}
h^{p,q}(\fun)=\dim\gr_{2p}^WH^{p+q}(U,\fun^{-1}(t)).
\end{equation}

If moreover $U$ is affine and $\fun$ is weakly tame, so that $H^{p+q}_{\DR}(U,\rd+\rd \fun)=0$ unless $p+q=n$, \cite[Cor.\,8.19]{S-Y14} gives, using the notation of Section \ref{sec:tame}:\footnote{The definition of $\delta_\gamma$ in \cite{S-Y14} should read $\dim\gr_\gamma^V(G_0(\fun)/uG_0(\fun))$.}
\[
h^{p,q}(\fun)=\begin{cases}
\dim\gr_{n-p}^V(G_0(\fun)/\theta G_0(\fun))=\dim\gr_F^p\gr_0^VG&\text{if }p+q=n,\\
0&\text{if }p+q\neq n,
\end{cases}
\]
and this is the number denoted by $f^{p,q}$ in \cite{K-K-P14}. In such a case, we have $H=H_1$ in the notation of Section \ref{sec:tame}\eqref{enum:tame5}.

The following criterion has been obtained, with a different approach of the irregular Hodge filtration, by Y.\,Shamoto.

\begin{proposition}[\cite{Shamoto17}]\label{prop:Shamoto}
Assume $U$ affine and $\fun$ weakly tame with unipotent monodromy operator at infinty. Then \eqref{eq:conjKKP} holds true if and only if the mixed Hodge structure of Section \ref{sec:tame}\eqref{enum:tame5} on $H=H_1$ is of Hodge-Tate type.
\end{proposition}

\begin{proof}
According to Lemma \ref{lem:CH}, proving the result amounts to identifying the space $\gr_0^VG$ endowed with its nilpotent operator $\rN$ with the space $H^n(U,\fun^{-1}(t))$ endowed with the nilpotent part of the (unipotent) monodromy (up to a nonzero constant). Choosing an extension $\Fun:\ccX\to\PP^1$ of $\fun$ as a projective morphism on a smooth variety $\ccX$ such that $\ccX\moins U$ is a divisor, and setting $\ccF=\bR j_*\CC_U$ ($j:U\hto\ccX$), we identify the dimension of $H^k(U,\fun^{-1}(t))$ with that of the $k$th-hypercohomology on $\ccX$ of the Beilinson extension $\Xi_\Fun\ccF$. Then the desired identification is given by \cite[Cor.\,1.13]{Bibi96a}.
\end{proof}

\subsection{The toric case of the conjecture of Katzarkov-Kontsevich-Pantev, first part}\label{sec:toric}
As usual in toric geometry, we denote by $M$ the lattice $\ZZ^n$ in $\CC^n$ and by $N$ its dual lattice. We fix a reflexive simplicial polyhedron $\Delta\subset\RR\otimes M$ with vertices in~$M$ and having $0$ in its interior (it is then the unique integral point in its interior), \cf\cite[\S4.1]{Batyrev94}. We denote by $\Deltas$ the dual polyhedron with vertices in $N$, which is also simplicial reflexive and has $0$ in its only interior point, and by $\Sigma\subset N$ the fan dual to $\Delta$, which is also the cone on $\Deltas$ with apex $0$. We assume that $\Sigma$ is the fan of nonsingular toric variety $X$ of dimension $n$, that is, each set of vertices of the same $(n-1)$-dimensional face of $\partial\Deltas$ is a $\ZZ$-basis of $N$. We know that
\begin{itemize}
\item
$X$ is Fano (\cite[Th.\,4.1.9]{Batyrev94}),
\item
the Chow ring $A^*(X)\simeq H^{2*}(X,\ZZ)$ is generated by the divisor classes $D_v$ corresponding to vertices $v\in V(\Deltas)$ of $\Deltas$, \ie primitive elements on the rays of $\Sigma$ (\cf \cite[p.\,101]{Fulton93}),
\item
we have $c_1(TX)=c_1(K_X^\vee)=\sum_{v\in V(\Deltas)}D_v$ in $H^{2*}(X,\ZZ)$ (\cf \cite[p.\,109]{Fulton93}), which satisfies Hard Lefschetz on $H^{2*}(X,\QQ)$, by ampleness of $K_X^\vee$.
\end{itemize}

Let us fix coordinates $x=(x_1,\dots,x_n)$ such that $\QQ[N]=\QQ[x,x^{-1}]$. We use the notation of Section \ref{sec:tame}\eqref{enum:tame6}. Due to the reflexivity of $\Deltas$, $L_\sigma$ has coefficients in~$\ZZ$ (it~corresponds to a vertex of $\Delta$). For $g\in\CC[x,x^{-1}]$, the $\sigma$-degree $\deg_\sigma(g)=\max_mL_\sigma(m)$ and the $\Deltas$-degree $\deg_{\Deltas}(g)=\max_\sigma\deg_\sigma(g)$ are thus nonnegative integers.

\begin{proposition}\label{prop:c1}
The case ``$f^{p,q}=h^{p,q}$'' of \cite[Conj.\,3.6]{K-K-P14} holds true if $\fun$ is the Laurent polynomial
\[
\fun(x)=\sum_{v\in V(\Deltas)}\hspace*{-3mm}x^v\in\QQ[x,x^{-1}].
\]
\end{proposition}

The idea of the proof is to notice that the property for the second morphism in Corollary \ref{cor:criterionHTLaurent} to be an isomorphism is exactly the property that $c_1(TX)$ satisfies the Hard Lefschetz property, and thus to identify its source and target as the cohomology of $X$ in suitable degree.

\begin{lemme}\label{lem:nondeg}
For $\Delta$ as above, any Laurent polynomial
\[
\fun_{\bma}(x)=\sum_{v\in V(\Deltas)}\hspace*{-3mm}a_vx^v\in\CC[x,x^{-1}],\quad \bma=(a_{v\in V})\in(\CC^*)^{V(\Delta^*)}.
\]
is convenient and non-degenerate in the sense of Kouchnirenko.
\end{lemme}

\begin{proof}
The Newton polyhedron of $\fun_{\bma}$ is equal to $\Deltas$, and $0$ belongs to its interior. In order to prove the non-degeneracy, we note that the vertices of any $(n-1)$-dimensional face $\sigma$ of $\partial\Deltas$ form a $\ZZ$-basis. It follows that, in suitable toric coordinates $y_1,\dots,y_n$, the restriction $\fun_{\bma|\sigma}$ can be written as $y_1+\cdots+y_n$, and the non-degeneracy is then obvious.
\end{proof}

\begin{proof}[Proof of Proposition \ref{prop:c1}]
Note that $\deg_{\Deltas}(\fun)=1$, as well as $\deg_{\Deltas}(x_i\partial\fun/\partial x_i)=1$. The Jacobian ring $\QQ[x,x^{-1}]/(\partial\fun)$ is endowed with the Newton filtration $\ccN_\bbullet$ induced by the $\Deltas$-degree $\deg_{\Deltas}$, and corresponds to $\ccN_\bbullet(G_0/\theta G_0)$ by multiplication by $\omega$. In the present setting, \cite[Th.\,1.1]{B-C-S05} identifies the graded ring $A^*(X)_\QQ$ with the graded ring
\[
\gr_\bbullet^\ccN\bigl(\QQ[x,x^{-1}]/(\partial\fun)\bigr).
\]
By applying Hard Lefschetz to $c_1(TX)$, we deduce that, for every $k\in\NN$ such that $0\leq k\leq[n/2]$, multiplication by the $(n-2k)$th power of the $\ccN$-class $[\fun]$ of $\fun$ induces an isomorphism
\[
[\fun]^{n-2k}:\gr_k^\ccN\bigl(\QQ[x,x^{-1}]/(\partial\fun)\bigr)\isom\gr_{n-k}^\ccN\bigl(\QQ[x,x^{-1}]/(\partial\fun)\bigr).
\]
By Corollary \ref{cor:criterionHTLaurent} for $H=H_1$, we deduce the assertion of the proposition from Proposition \ref{prop:Shamoto}.
\end{proof}

\subsection{The toric case of the conjecture of Katzarkov-Kontsevich-Pantev, second part}

We now prove the main result of this note.

\begin{theoreme}\label{th:KKPconj}
The case ``$f^{p,q}=h^{p,q}$'' of \cite[Conj.\,3.6]{K-K-P14} holds true for any Laurent polynomial
\[
\fun_{\bma}(x)=\sum_{v\in V(\Deltas)}\hspace*{-3mm}a_vx^v\in\CC[x,x^{-1}],\quad \bma=(a_{v\in V})\in(\CC^*)^{V(\Delta^*)}.
\]
\end{theoreme}

\begin{remarque}
The case where $n=3$ was already proved differently by Y.\,Shamoto \cite[\S4.2]{Shamoto17}.
\end{remarque}

\begin{proof}
Let us set $H(\fun_{\bma})=H_1(\fun_{\bma})=\gr_0^VG(\fun_{\bma})$, where $G(\fun_{\bma})$ is the localized Laplace transform of the Gauss-Manin system for $\fun_{\bma}$ as in Section \ref{sec:tame}\eqref{enum:tame2}. By Lemma \ref{lem:nondeg}, we can apply the results of Section~\ref{sec:tame} to $\fun_{\bma}$ for any $\bma\in(\CC^*)^{V(\Delta^*)}$. We will prove that, for fixed $p$, both terms $\dim\gr_{n-p}^GH(\fun_{\bma})$ and $\dim\gr_{2p}^WH(\fun_{\bma})$ in Lemma \ref{lem:CH} are independent of $\bma$. Since they are equal if $\bma=(1,\dots,1)$, after Proposition \ref{prop:c1}, they are equal for any $\bma\in(\CC^*)^{V(\Delta^*)}$, as wanted.

\begin{enumerate}
\item\label{enum:pfth1}
For the first term, we will use \cite{N-S97}. We have denoted there $\dim\gr_p^GH(\fun_{\bma})$ by $\nu_p(\fun_{\bma})$ and, since $\gr_\alpha^VG=0$ for $\alpha\notin\ZZ$, it is also equal to the number denoted there by $\Sigma_{p-1}(\fun_{\bma})$. By the theorem in \cite{N-S97} and Lemma \ref{lem:nondeg}, $\Sigma_{p-1}(\fun_{\bma})$ depends semi-continuously on $\bma$. On the other hand, according to \cite{Kouchnirenko76}, $\dim H(\fun_{\bma})$ is independent of $\bma$ and is computed only in terms of $\Deltas$. Since $\dim H(\fun_{\bma})=\sum_p\Sigma_{p-1}(\fun_{\bma})$, each term in this sum is also constant with respect to $\bma$.
\item\label{enum:pfth2}
We will prove the local constancy of $\dim\gr_{2p}^WH(\fun_{\bma})$ near any $\bma_o\in(\CC^*)^{V(\Delta^*)}$. As noticed in \cite[\S4]{D-S02a}, we can apply the results of Section 2 of \loccit\ to $\fun_{\bma_o}$. We fix a Stein open set $\ccB^o$ adapted to $\fun_{\bma_o}$ as in \cite[\S2a]{D-S02a}, and fix a neighbourhood~$X$ of~$\bma_o$ so that it is also adapted to any $\fun_{\bma}$ for $\bma$ in this neighbourhood. By construction, all the critical points of $\fun_{\bma_o}$ are contained in the interior of $\ccB^o$ if $X$ is chosen small enough, and since $\mu(\fun_{\bma})$ is constant, the same property holds for $\bma\in X$. By using successively Theorem 2.9, Remark 2.11 and Proposition 1.20(1) in \cite{D-S02a}, we deduce that, when $\bma$ varies in $X$, the localized partial Laplace transformed Gauss-Manin systems $G(\fun_{\bma})$ form an $\cO_X[\tau,\tau^{-1}]$-free module with integrable connection and regular singularity along $\tau=0$, which is compatible with base change with respect to $X$. As~a~consequence, the monodromy of each $G(\fun_{\bma})$ around $\tau=0$ is constant, and the assertion follows.\qedhere
\end{enumerate}
\end{proof}

\begin{remarque}[suggested by the referee]
If we relax the condition in Section \ref{sec:toric} that the toric Fano variety $X$ is \emph{nonsingular}, then we have to consider the orbifold Chow ring  of $X$ as in \cite{B-C-S05}, or the Chen-Ruan orbifold cohomology of $X$. For the cohomology of the untwisted sector (\ie the usual cohomology), the Hard Lefschetz theorem is still valid (\cf \cite{Steenbrink77}) and Proposition \ref{prop:c1} still holds, \ie \eqref{eq:conjKKP} holds for $\fun$. Moreover, Part \eqref{enum:pfth2} of the proof of Theorem \ref{th:KKPconj} also extends to this setting. However, the semicontinuity result of \cite{N-S97} used in Part \eqref{enum:pfth1} of the proof is not enough to imply the constancy (with respect to $\bma$) of $\nu_p(\fun_{\bma})$.

On the other hand, one can also consider the various $h^{p,q}_\alpha(\fun)$ for $\alpha\in(0,1)\cap\QQ$ and, correspondingly, the twisted sectors of the orbifold $X$. In such a case, Hard Lefschetz for $\fun$ may already give trouble (\cf \cite{Fernandez06}).
\end{remarque}

\backmatter
\providecommand{\eprint}[1]{\href{http://arxiv.org/abs/#1}{\texttt{arXiv\string:\allowbreak#1}}}\providecommand{\doi}[1]{\href{http://dx.doi.org/#1}{\texttt{doi\string:\allowbreak#1}}}
\providecommand{\bysame}{\leavevmode ---\ }
\providecommand{\og}{``}
\providecommand{\fg}{''}
\providecommand{\smfandname}{\&}
\providecommand{\smfedsname}{eds.}
\providecommand{\smfedname}{ed.}

\end{document}